\documentclass[12pt,a4paper]{amsart}
\usepackage{enumerate,amsmath,amsthm,amssymb,cite,mathrsfs,graphicx,caption,subfigure}
\usepackage{float}

\newcommand{\A}{\mathcal{A}}

\DeclareMathOperator{\RE}{Re} \DeclareMathOperator{\IM}{Im}

\numberwithin{equation}{section}
\newtheorem{theorem}{Theorem}[section]
\newtheorem{lemma}[theorem]{Lemma}
\newtheorem{corollary}[theorem]{Corollary}

\newtheorem{example}[theorem]{Example}
\theoremstyle{remark}

\allowdisplaybreaks

\makeatletter
\@namedef{subjclassname@2020}{\textup{2020} Mathematics Subject Classification}
\makeatother
\begin{document}
\title[Briot--Bouquet differential subordination and Bernardi's integral operator]{Briot--Bouquet differential subordination and Bernardi's integral operator}

\dedicatory{Dedicated to Prof.\ Dato' Indera Rosihan M. Ali}

\author[K. Sharma]{Kanika sharma}
\address{Department of Mathematics, Atma Ram Sanatan Dharma College, University of Delhi, Delhi--110 021, India}
\email{ksharma@arsd.du.ac.in; kanika.divika@gmail.com}

\author{Rasoul Aghalary}
\address{Department of Mathematics, Faculty of Science, Urmia University, Urmia, Iran}
\email{raghalary@yahoo.com; r.aghalary@urmia.ac.ir}

\author{V. Ravichandran}
\address{Department of Mathematics, National Institute of Technology, Tiruchirappalli--620 015, India} \email{ravic@nitt.edu; vravi68@gmail.com}

\begin{abstract}  The conditions on $A$, $B$, $\beta$ and $\gamma$ are obtained for an analytic function $p$ defined on the open unit disc $\mathbb{D}$ and  normalized by $p(0)=1$  to be subordinate to $(1+Az)/(1+Bz)$, $-1\leq B<A \leq 1$ when $p(z)+ zp'(z)/(\beta p(z)+\gamma)$ is subordinate to $e^{z}$. The conditions on these parameters are derived for the function $p$ to be subordinate to $\sqrt{1+z}$ or $e^{z}$ when $p(z)+ zp'(z)/(\beta p(z)+\gamma)$ is subordinate to $(1+Az)/(1+Bz)$. The conditions on $\beta$ and $\gamma$ are determined for the function $p$ to be subordinate to $e^{z}$ when $p(z)+ zp'(z)/(\beta p(z)+\gamma)$ is subordinate to $\sqrt{1+z}$. Related result for the function $p(z)+ zp'(z)/(\beta p(z)+\gamma)$ to be in the parabolic region bounded by the $\RE w=|w-1|$ is investigated. Sufficient conditions for the Bernardi's integral operator to belong to the various subclasses of starlike functions are obtained as applications.

\end{abstract}

\maketitle
\section{introduction}
Let $\mathcal{H}$ denote the class of analytic functions in the unit disc $\mathbb{D}$. For a natural number $n$, let $\mathcal{H}[a,n]$ be the subset of $\mathcal{H}$ consisting of functions $p$ of the form $p(z)=a+p_{n}z^{n}+p_{n+1}z^{n+1}+\cdots$. Suppose that $h$ is a univalent function defined on $\mathbb{D}$ with $h(0)=a$ and the function $p \in \mathcal{H}[a,n]$. The Briot--Bouquet differential subordination is the first order differential subordination of the form
\begin{equation}\label{p7eqn1.8}
p(z)+\dfrac{zp'(z)}{\beta p(z)+ \gamma} \prec h(z),
\end{equation}
where $\beta \neq 0, \gamma \in \mathbb{C}$. This particular differential subordination has many interesting applications in the theory of univalent functions. Ruschewyh and Singh \cite{rus1} proved that if the function $p \in \mathcal{H}[1,1]$, $\beta>0, \RE \gamma \geq 0$ and $h(z)=(1+z)/(1-z)$ in ~\eqref{p7eqn1.8} and the function $q \in \mathcal{H}$ satisfy the  differential equation
\[q(z)+\dfrac{zp'(z)}{\beta p(z)+ \gamma}=\frac{1+z}{1-z},\]
then $\min_{|z|=r} \RE p(z) \geq \min_{|z|=r} \RE q(z).$ More related results are proved in \cite{mill1, mill, een}.
For $c>-1$ and $f \in \mathcal{H}[0,1]$, the function $F \in \mathcal{H}[0,1]$ given by Bernardi's integral operator is defined as
\begin{equation}\label{p7eqn2.11}
F(z)=\frac{c +1}{z^{c}}\int^{z}_{0}t^{c-1}f(t)dt.																																	
\end{equation}
There is an important connection between Briot--Bouquet differential equations and the Bernardi's integral operator. If we set $p(z)=zF'(z)/F(z),$ where $F$ is given by  ~\eqref{p7eqn2.11}, then the functions $f$ and $p$ are related through the following Briot--Bouquet differential equation
\[ \frac{zf'(z)}{f(z)}=p(z)+\frac{zp'(z)}{p(z)+c}.\]

Several authors have investigated results on Briot--Bouquet differential subordination. For example, Ali \textit{et al.}\cite{ali4} determined the conditions on $A, B, D$ and $E$ for $p(z)\prec (1+Az)/(1+Bz)$ when $p(z)+ zp'(z)/(\beta p(z)+\gamma)$ is subordinate to $(1+Dz)/(1+Ez)$, $(A,B,D,E \in [-1,1])$. For related results, see \cite{een, mill, mill1, rus1}. Recently, Kumar and Ravichandran ~\cite{sush1} obtained the conditions on $\beta$ so that $p(z)$ is subordinate to $e^{z}$ or $(1+Az)/(1+Bz)$ whenever $1+ \beta p(z)/p'(z)$ is subordinate to $\sqrt{1+z}$ or $(1+Az)/(1+Bz),$ $(-1 \leq B <A \leq 1)$. We investigate generalised problems for regions that were considered recently by many authors. In Section ~\ref{p7sec2}, we find conditions on $\gamma$ and $\beta$ so that $p(z)+ zp'(z)/(\beta p(z)+\gamma)$ is subordinate to $\sqrt{1+z}$ implies $p(z)\prec e^{z}$. Conditions on $A, B, \beta$ and $\gamma$ are also determined so that $p(z)+ zp'(z)/(\beta p(z)+\gamma) \prec(1+Az)/(1+Bz)$ implies $p(z)\prec \sqrt{1+z}$ or $e^{z}$. We determine conditions on $A, B, \beta$ and $\gamma$ so that $p(z)\prec (1+Az)/(1+Bz), (-1\leq B<A\leq1)$ when $p(z)+ zp'(z)/(\beta p(z)+\gamma) \prec e^{z}$ or $\varphi_{PAR}(z)$.
 The function $\varphi_{PAR}:\mathbb{D}\to\mathbb{C}$ is given by
 \begin{equation}\label{p7eqn2.9}
\varphi_{PAR}(z):=1 + \frac{2}{\pi^{2}} \left(\log \frac{1+\sqrt{z}}{1-\sqrt{z}}\right)^{2},\quad \IM\sqrt{z}\geq 0
\end{equation}
and  $\varphi_{PAR}(\mathbb{D})=\left\{w=u+iv:v^{2}<2u-1\right\}=\left\{w:\RE w>|w-1| \right\}=:\Omega_{P}$.
 As an application of our results, we give sufficient conditions for the Bernardi's integral operator to belong to the various subclasses of starlike functions which we define below.

Let $\mathcal{A}$ be the class of all functions $f\in\mathcal{H}$ normalized by the conditions $f(0)=0$ and $f'(0)=1$. Let $\mathcal{S}$ denote the subclass of $\mathcal{A}$ consisting of univalent (one-to-one) functions. For an analytic function $\varphi$ with $\varphi(0)=1$, let
\[ \mathcal{S}^*(\varphi) := \left\{ f\in\mathcal{A}:\frac{					 zf'(z)}{f(z)}\prec\varphi(z)\right\}.\]
This class unifies various classes of starlike functions when $\RE \varphi>0$. Shanmugam \cite{shan} studied the convolution properties of this class when $\varphi$ is convex while Ma and Minda \cite{mamin2} investigated the growth, distortion and coefficient estimates under less restrictive assumption that $\varphi$ is starlike and $\varphi(\mathbb{D})$ is symmetric with respect to the real axis. Notice that, for $-1\leq B<A\leq 1$, the class $\mathcal{S}^*[A,B]:=\mathcal{S}^*((1+Az)/(1+Bz))$ is the class of Janowski starlike functions \cite{jano,pot}. For $0\leq \alpha <1$, the class $\mathcal{S}^*[1-2\alpha, -1]=:\mathcal{S}^*(\alpha)$ is the familiar class of starlike functions of order $\alpha$, introduced by Robertson \cite{rob}. The class $\mathcal{S}^*:=\mathcal{S}^*(0)$ is the class of starlike function. The class $\mathcal{S}_{P}:=\mathcal{S}^*(\varphi_{PAR})$ is the class of parabolic starlike functions, introduced by R{\o}nning \cite{ron}, consists of function $f\in\mathcal{A}$ satisfying
\[\RE\left(\frac{z f'(z)}{f(z)}\right)>\left|\frac{z f'(z)}{f(z)}-1\right|,\quad z\in\mathbb{D}. \]
Sok\'{o}l and Stankiewicz \cite{sokol96} have introduced and studied the class $\mathcal{S}^*_L:=\mathcal{S}^*(\sqrt{1+z});$ the class $\mathcal{S}^*_L$ consists of functions $f\in \mathcal{A}$ such that $zf'(z)/f(z)$ lies in the region bounded by the right-half of the lemniscate of Bernoulli given by $\Omega_{L} :=\left\{w\in\mathbb{C}:|w^{2}-1|<1 \right\}$. Another class $\mathcal{S}^*_e:= \mathcal{S}^*(e^{z})$, introduced recently by Mendiratta \textit{et al.} \cite{men1}, consists of functions $f\in\A$ satisfying the condition $|\log(zf'(z)/f(z))|<1$. There has been several works \cite{jano1,nav,vraviron,vravicmft,sokol09,sokol09b,sharma,sharma1,sharma2,sharma3,sharma4,sharma5,jain} related to these classes.

The following results are required in our investigation.

\begin{lemma}\cite[Theorem 2.1, p.2]{adiba}\label{p7lem2}
	Let $\Omega\subset \mathbb{C}$ and suppose that $\psi: {\mathbb{C}}^{2} \times \mathbb{D} \to \mathbb{C}$ satisfies the condition $\psi(e^{{e}^{it}}, k e^{it}e^{{e}^{it}}; z) \notin \Omega$, where $z \in \mathbb{D}$, $t \in [0, 2\pi]$ and $k \geq 1$. If $p \in \mathcal{H}[1,1]$ and $\psi(p(z),zp'(z);z)\in \Omega$ for $z \in \mathbb{D}$, then $p(z) \prec e^{z}$ in $\mathbb{D}$.
\end{lemma}

\begin{lemma}\cite[Lemma 1.3, p.28]{rus}\label{p7lem1}
Let $w$ be a meromorphic function in $\mathbb{D}$, $w(0)=0$. If for some $z_{0}\in \mathbb{D}$, $\max_{|z|\leq |z_{0}|}|w(z)|=|w(z_{0})|$, then it follows that $z_{0}w'(z_{0})/w(z_{0})\geq 1.$
\end{lemma}

\section{Briot--Bouquet differential subordination}\label{p7sec2}

In the first result, we find conditions on the real numbers $\beta$ and $\gamma$ so that $p(z)\prec e^z$, whenever $p(z)+(zp'(z))/(\beta p(z)+\gamma)\prec \sqrt{1+z}$, where $p \in \mathcal{H}$ with $p(0)=1.$ This result gives the sufficient condition for $f \in \mathcal{A}$ to belong to the class $\mathcal{S}^*_e$ by substituting $p(z)=z f'(z)/f(z).$

\begin{theorem}\label{p7thm2.6}
	Let $\beta, \gamma \in \mathbb{R}$ satisfying $\max\{-\gamma/e,-\gamma e+e/(1-\sqrt{2}e)\} \leq \beta \leq -e \gamma$. Let $p \in \mathcal{H}$ with $p(0)=1$. If the function $p$ satisfies
	\[ p(z)+\frac{zp'(z)}{\beta p(z)+\gamma} \prec \sqrt{1+z}, \]
	then $p(z)\prec e^{z}$.
\end{theorem}

\begin{proof}
Define the functions $\psi: {\mathbb{C}}^{2} \times \mathbb{D} \to \mathbb{C}$ and $q:\mathbb{D} \to \mathbb{C}$ as follows:
\begin{equation}\label{p7eqn2.56}
	\psi(r,s;z)=r+\frac{s}{\beta r+\gamma}\quad \text{and}\quad q(z)=\sqrt{1+z}
	\end{equation}
so that  $\Omega:=q(\mathbb{D})=\left\{w\in\mathbb{C}:|w^{2}-1|<1 \right\}$ and $\psi(p(z),zp'(z);z)\in \Omega$ for $z \in \mathbb{D}$. To prove $p(z)\prec e^{z}$, we use Lemma ~\ref{p7lem2} so we need to show that $\psi(e^{{e}^{it}}, k e^{it}e^{{e}^{it}}; z) \notin \Omega$ which is equivalent to show that $|(\psi(e^{{e}^{it}}, k e^{it}e^{{e}^{it}}; z))^{2}-1| \geq 1$, where $z \in \mathbb{D}$, $t \in [-\pi, \pi]$ and $k \geq 1$. A simple computation and ~\eqref{p7eqn2.56} yield that
\begin{equation*}
	\psi(e^{{e}^{it}}, k e^{it}e^{{e}^{it}}; z)=e^{e^{i t}}+\frac{k e^{i t}e^{e^{i t}}}{\beta e^{e^{i t}}+\gamma}\quad (-\pi\leq t \leq \pi)
	\end{equation*}
and
\begin{equation}\label{p7eqn2.20}
	|(\psi(e^{{e}^{it}}, k e^{it}e^{{e}^{it}}; z))^{2}-1|^{2}=:\frac{f(t)}{g(t)}\quad (-\pi\leq t \leq \pi),
	\end{equation}
	where
	\begin{equation*}
	\begin{split}
	f(t)=&\big(e^{2 \cos t} \cos(2\sin t)((\gamma +k\cos t+\beta e^{\cos t}\cos(\sin t))^2-(k\sin t+\beta \sin(\sin t)e^{\cos t})^2)\\
	&-2\sin(2\sin t)e^{2\cos t}(k \sin t+\beta \sin(\sin t)e^{\cos t})(\gamma+k\cos t+\beta e^{\cos t}\cos(\sin t))\\
	&+\beta^2 \sin^2(\sin t)e^{2\cos t}-(\gamma+\beta e^{\cos t}\cos (\sin t))^2\big)^2+\big(2 e^{2 \cos t}\cos(2\sin t)(k \sin t\\
	&+\beta \sin (\sin t)e^{\cos t})(\gamma +k \cos t+\beta e^{\cos t}\cos(\sin t))+\sin(2 \sin t)e^{2\cos t}((\gamma+k\cos t\\
	&+\beta e^{\cos t}\cos(\sin t))^2-(k \sin t+\beta \sin(\sin t)e^{\cos t})^2)\\
	&-2\beta \sin(\sin t)e^{\cos t}(\gamma+\beta e^{\cos t}\cos(\sin t))\big)^2
	\end{split}
	\end{equation*}
	and
	\begin{equation*}
	g(t)=(\beta ^2 \sin ^2(\sin t) e^{2 \cos t}+(\gamma +\beta  e^{\cos t} \cos (\sin t))^2)^2.
	\end{equation*}
	Define the function $h:[-\pi,\pi]\to \mathbb{R}$ by $h(t)=f(t)-g(t)$. Since $h(-t)=h(t)$, we restrict to $0\leq t \leq \pi$. It can be easily verified that the function $h$ attains its minimum value either at $t=0$ or $t=\pi$. For $k \geq 1$, we have
	\begin{equation}\label{p7eqn2.21}
	h(0)=(e^2 (e \beta +\gamma +k)^2-(e \beta +\gamma )^2)^2-(e \beta +\gamma )^4
	\end{equation}
	and
	\begin{equation}\label{p7eqn2.22}
	h(\pi)=\Big(\Big(\frac{\beta/e+\gamma -k}{e}\Big)^2-\Big(\frac{\beta}{e}+\gamma \Big)^2\Big)^2-\Big(\frac{\beta }{e}+\gamma \Big)^4.
	\end{equation}
	The given relation $\beta \geq -\gamma/e$ gives $e\beta+\gamma \geq 0$ so that $e(k+e\beta+\gamma)>\sqrt{2}(e\beta+\gamma)$ which implies $e^2(k+e\beta+\gamma)^2-(e\beta+\gamma)^2>(e\beta+\gamma)^2$. Thus, the use of ~\eqref{p7eqn2.21} yields $h(0)>0$.
	
	The given condition $1/(1-\sqrt{2}e)\leq \gamma+\beta/e \leq 0$ leads to $(\gamma+\beta/e)(1-\sqrt{2}e) \leq 1$ which gives that $-k+\gamma+\beta/e \leq -1+\gamma+\beta/e \leq \sqrt{2}e(\gamma+\beta/e)$ which implies $((-k+\gamma+\beta/e)/e)^2 \geq 2(\gamma+\beta/e)^2$ which further implies $((-k+\gamma+\beta/e)/e)^2 -(\gamma+\beta/e)^2 \geq (\gamma+\beta/e)^2$. Hence, by using ~\eqref{p7eqn2.22}, we get that $h(\pi)\geq 0.$ So, $h(t)\geq 0, (0\leq t \leq \pi)$ and thus, ~\eqref{p7eqn2.20} implies $|(\psi(e^{{e}^{it}}, k e^{it}e^{{e}^{it}}; z))^{2}-1| \geq 1$ and therefore $p(z) \prec e^{z}$.
\end{proof}

We will illustrate Theorem ~\ref{p7thm2.6} by the following example:

\begin{example}\label{p7ex2.6}
	By taking $\beta=1$ and $\gamma=c$ $(c>-1)$ in Theorem ~\ref{p7thm2.6}, we get $-1/e+1/(1-\sqrt{2}e)\leq c \leq -1/e$. By taking $\beta=1$, $-1/e+1/(1-\sqrt{2}e)\leq \gamma \leq -1/e$, $n=1$, $h(z)=\sqrt{1+z}$, $a=1$ in \cite[Theorem 3.2d, p.86]{ural}, we get $\RE(a \beta +\gamma)>0$ and $\beta h(z)+\gamma \prec R_{a \beta+\gamma, n}(z)$, where $R_{d,f}(z)$ is the open door mapping given by $R_{d,f}(z):=d(1+z)/(1-z)+(2fz)/(1-z^{2})$. Thus by the use of \cite[Theorem 3.2d, p.86]{ural}, we get
	\[p(z)= -\gamma +\int_0^1 \frac{t^{-\gamma } e^{2 \sqrt{z+1}-2 \sqrt{t z+1}} \left(\sqrt{t z+1}+1\right)^2}{\left(\sqrt{z+1}+1\right)^2} \, dt\]
	which satisy the equation $p(z)+zp'(z)/(\beta p(z)+\gamma)=h(z)$. Then $p(z) \prec e^{z}$.
\end{example}	

Suppose that the function $F$ be given by Bernardi's integral ~\eqref{p7eqn2.11}. Now we discuss the sufficient conditions for the function $F$ to belong to various subclasses of starlike functions. We will illustrate the Theorem ~\ref{p7thm2.6} by the following corollary.

\begin{corollary}\label{p7cor3.5}
	\begin{enumerate}[(i)]
		\item  If the function $f \in \mathcal{S}^*_L$ and the conditions of the Theorem ~\ref{p7thm2.6} hold with $\beta=1$ and $\gamma=c$, then $F \in \mathcal{S}^*_e.$
		\item If the function $f'(z) \prec \sqrt{1+z}$ and the conditions of the Theorem ~\ref{p7thm2.6} hold with $\beta=0$ and $\gamma=c+1$, then  $F'(z)\prec e^z.$
	\end{enumerate}
\end{corollary}

\begin{proof}
$(i)$ Let the function $p:\mathbb{D}\to\mathbb{C}$ be defined by $p(z)=zF'(z)/F(z).$ Then $p$ is analytic in $\mathbb{D}$ with $p(0)=1$. Upon differentiating Bernardi's integral given by ~\eqref{p7eqn2.11}, we obtain
\begin{equation}\label{p7eqn2.12}
(c+1)f(z)=zF'(z)+cF(z).
\end{equation}
A computation  now yields
\[ \frac{zf'(z)}{f(z)}=p(z)+\frac{zp'(z)}{p(z)+c}.\]
By taking $\beta=1$ and $\gamma=c$, the first part of the corollary follows from Theorem ~\ref{p7thm2.6}.

$(ii)$ By defining a function $p$ by  $p(z)=F'(z)$ and using ~\eqref{p7eqn2.12}, we get
\[ f'(z)=\frac{zF''(z)}{c+1}+F'(z). \]
By taking $\beta=0$ and $\gamma=c+1$, the result follows from Theorem ~\ref{p7thm2.6}.
\end{proof}

In the following result, we derive   conditions on the real numbers $A, B$, $\beta$ and $\gamma$ so that $p(z)+(zp'(z))/(\beta p(z)+\gamma) \prec e^{z}$ implies $p(z)\prec (1+Az)/(1+Bz), (-1\leq B<A\leq 1),$ where $p \in \mathcal{H}$ with $p(0)=1.$ This result gives the sufficient condition for $f \in \mathcal{A}$ to belong to the class $\mathcal{S}^*[A,B]$ by substituting $p(z)=z f'(z)/f(z).$

\begin{theorem}\label{p7thm2.2}
	Let $-1< B<A\leq 1$ and $\beta, \gamma \in \mathbb{R}$. Suppose that
	\begin{enumerate}[(i)]
		\item $\big(A-B\big)/\big((1\mp B)((1\mp A)\beta+(1 \mp B)\gamma)\big)\geq \pm (1 \mp A)/(1\mp B) + e.$
		\item $\beta(1\pm A)+\gamma(1 \pm B)>0$.
	\end{enumerate}
	Let $p \in \mathcal{H}$ with $p(0)=1$. If the function $p$ satisfies
	\[ p(z)+\frac{zp'(z)}{\beta p(z)+\gamma} \prec e^{z}, \]
	then $p(z)\prec (1+Az)/(1+Bz)$.
\end{theorem}

\begin{proof}
	Define the functions $P$ and $w$ as follows:
	\begin{equation}\label{p7eqn2.4}
	P(z)=p(z)+\frac{zp'(z)}{\beta p(z)+\gamma}\quad \text{and}\quad w(z)=\frac{p(z)-1}{A-Bp(z)}
	\end{equation}
	so that  $p(z)=(1+Aw(z))/(1+Bw(z)).$ Clearly, $w(z)$ is analytic in $\mathbb{D}$ with $w(0)=0$. In order to prove $p(z)\prec (1+Az)/(1+Bz)$, we need to show that $|w(z)|<1$ in $\mathbb{D}$. If possible, suppose that there exists $z_{0}\in \mathbb{D}$ such that
	\[\max_{|z|\leq |z_{0}|}|w(z)|=|w(z_{0})|=1,\]
	then by Lemma ~\ref{p7lem1}, it follows that there exists $k\geq 1$ so that $z_{0}w'(z_{0})=kw(z_{0}).$ Let $w(z_{0})=e^{it}$, $(-\pi\leq t \leq \pi)$ and $G:=A\beta +B \gamma$. A simple calculation and by using ~\eqref{p7eqn2.4}, we get
	\begin{equation}\label{p7eqn2.5}
	P(z_{0})=\frac{k e^{i t} (A-B)+\left(1+A e^{i t}\right) \left(\beta +\gamma +G e^{i t}\right)}{\left(1+B e^{i t}\right) \left(\beta +\gamma +G e^{i t}\right)}=:u+iv \quad (-\pi \leq t \leq \pi).
	\end{equation}
We derive  a contradiction by  showing  $|\log P(z_{0})|^{2}\geq 1$. This later inequality is equivalent to
	\begin{equation}\label{p7eqn2.6}
	f(t):=4 (\arg (u+i v))^2+(\log \left(u^2+v^2\right))^2-4 \geq 0 \quad (-\pi \leq t \leq \pi).
	\end{equation}
	From ~\eqref{p7eqn2.5}, we get
	\begin{equation*}
	\begin{split}
	u&=\frac{1}{\left(B^2+2 B \cos t+1\right) \left((\beta +\gamma )^2+G^2+2 G (\beta +\gamma ) \cos t\right)}\big(G (A+B) (\beta +\gamma ) \cos 2t\\
	&\quad{}+\cos t \big(A (B G (2 (\beta +\gamma )+k)+G^2+(\beta +\gamma ) (\beta +\gamma +k))-B^2 G k +2 G (\beta +\gamma )\\
	&\quad{}+B\left(G^2-(\beta +\gamma ) (-\beta -\gamma +k)\right)\big)+(\beta +\gamma ) \left(A B (\beta +\gamma +k)+\beta -B^2 k+\gamma \right)\\
	&\quad{}+G^2 (A B+1)+G (A (\beta +\gamma +k)+B (\beta +\gamma -k))\big)
	\end{split}
	\end{equation*}
	and
	\[v=\frac{(A-B) \sin t \left(-B G k+G^2+2 G (\beta +\gamma ) \cos t+(\beta +\gamma ) (\beta +\gamma +k)\right)}{\left(B^2+2 B \cos t+1\right) \left((\beta +\gamma )^2+G^2+2 G (\beta +\gamma ) \cos t\right)}.\]
	Substituting these values of $u$ and $v$ in ~\eqref{p7eqn2.6}, we observe that $f(t)$ is an even function of $t$ and so, it is enough to show that $f(t) \geq 0$ for $t\in[0,\pi]$. It can be easily verified that the function $f(t)$ attains its minimum value either at $t=0$ or $t=\pi$. We show that both $f(0)$ and $f(\pi)$ are non negative. Note that, for $k \geq 1$,
	\begin{equation}\label{p7eqn2.7}
	f(0)=-4+4(\arg \psi(k))^2 +(\log( \psi^{2}(k)))^2
	\end{equation}
	and
	\begin{equation}\label{p7eqn2.8}
	f(\pi)=-4+4(\arg (-\phi(k)))^2 +(\log( \phi^{2}(k)))^2,
	\end{equation}
	where $\psi(k):=\big(A^2 \beta +A (2 \beta +B \gamma +\gamma +k)+\beta +B (\gamma -k)+\gamma\big)/\big((1+B) (\beta(1+A) + \gamma(1 +B ))\big)$ and $\phi(k):=\big(A^2 \beta -2 A \beta +(A-1) (B-1) \gamma -A k+\beta +B k\big)/\big((B-1) (-A \beta +\beta -B \gamma +\gamma )\big)$. The function $\psi$ is increasing as $\psi'(k)=(A-B)/\big((1+B) (\beta(1+A) + \gamma(1 +B ))\big) >0$ using the given condition $(ii)$ and therefore, the given hypothesis $(i)$ yields that $\psi(k)\geq \psi(1)=(1+A)/(1+B) + (A-B)/\big((1+B) (\beta(1+A) + \gamma(1 +B ))\big) \geq e$ which gives that $\arg \psi(k)=0$ and $(\log( \psi^{2}(k)))^2\geq (2\log e)^{2}=4.$ Thus, the use of ~\eqref{p7eqn2.7} yields $f(0)\geq 0.$
	
	The function $\phi$ is increasing as $\phi'(k)=(A-B)/\big((1-B) (\beta(1-A) + \gamma(1- B ))\big) >0$ using the given condition $(ii)$ and therefore, the given hypothesis $(i)$ yields that $\phi(k)\geq \phi(1)=-(1-A)/(1-B) + (A-B)/\big((1-B) (\beta(1-A) + \gamma(1 -B ))\big) \geq e$ which further implies $\arg (-\phi(k))=\pi$ and $(\log( \phi^{2}(k)))^2\geq (2\log e)^{2}=4.$ Hence, by using ~\eqref{p7eqn2.8}, we get $f(\pi)\geq 4\pi^{2}>0.$ This completes the proof.
\end{proof}

We will illustrate Theorem ~\ref{p7thm2.2} by the following example:

\begin{example}\label{p7ex2.2}
	By taking $A=1/2$, $B=-1/2$, $\beta=1$ and $\gamma=c$ $(c>-1)$ in Theorem ~\ref{p7thm2.2}, we get $-1/3 \leq c \leq (1-e)/(1+3e)$. By taking $\beta=1$, $-1/3 \leq \gamma \leq (1-e)/(1+3e)$, $n=1$, $h(z)=e^{z}$, $a=1$ in \cite[Theorem 3.2d, p.86]{ural}, we get $\RE(a \beta +\gamma)>0$ and $\beta h(z)+\gamma \prec R_{a \beta+\gamma, n}(z)$, where $R_{d,f}(z)$ is the open door mapping given by $R_{d,f}(z):=d(1+z)/(1-z)+(2fz)/(1-z^{2})$. Thus by the use of \cite[Theorem 3.2d, p.86]{ural}, we get
	\[p(z)= \int_0^1 t^{1-\gamma } e^{-\text{Chi}(t z)+\text{Chi}(z)-\text{Shi}(t z)+\text{Shi}(z)} \, dt-\gamma\]
	which satisy the equation $p(z)+zp'(z)/(\beta p(z)+\gamma)=h(z)$. Then $p(z) \prec (2+z)/(2-z)$. Here, $Chi(z)$ and $Shi(z)$ are the hyperbolic cosine integral function and the hyperbolic sine integral function respectively defined as follows:
	\[ Chi(z)= \eta + \log (z) + \int_0^z \frac{\cosh(t)-1}{t} dt \quad \text{and} \quad Shi(z)=\int_0^z \frac{\sinh(t)}{t} dt, \]
	where $\eta$ is the Euler's constant.
\end{example}

The next corollary is obtained by substituting $p(z)=zf'(z)/f(z)$ with $\gamma=0$, $B=0$ and $A=1-\alpha$, $(0 \leq \alpha<1)$ in Theorem ~\ref{p7thm2.2}.

\begin{corollary}\label{p7cor3.12}
Let $0 \leq \alpha<1$ and $\beta>0$ satisfy the conditions $\alpha+e+\beta^{-1} \leq (\alpha \beta)^{-1}$ and $1-\alpha \geq \beta(2-\alpha)(e-2+\alpha)$. If the function $f \in \mathcal{A}$ satisfies the subordination
\[\frac{z f'(z)}{f(z)} + \frac{1}{\beta}\left(1+\frac{zf''(z)}{f'(z)}-\frac{zf'(z)}{f(z)}\right)\prec e^{z},\]
then $f \in \mathcal{S}^{*}_{\alpha}.$
\end{corollary}

Our next corollary deals with the class $ \mathcal{R}[A,B]$ defined by
\[ \mathcal{R}[A,B]=\left\{ f\in\mathcal{A}: f'(z)\prec \frac{1+Az}{1+Bz}\right\}.\]
The two parts of the following corollary are obtained by taking $p(z)=zF'(z)/F(z)$ with $\beta=1$, $\gamma=c$ and $p(z)=F'(z)$ with $\beta=0$, $\gamma=c+1$ respectively in Theorem ~\ref{p7thm2.2}.

\begin{corollary}\label{p7cor3.5}
	\begin{enumerate}[(i)]
		\item  If the function $f \in \mathcal{S}^*_e$ and the conditions of the Theorem ~\ref{p7thm2.2} hold with $\beta=1$ and $\gamma=c$, then $F \in \mathcal{S}^*[A,B].$
		\item The function $f'(z) \prec e^{z}$ and the conditions of the Theorem ~\ref{p7thm2.2} hold with $\beta=0$ and $\gamma=c+1$, then  $F\in \mathcal{R}[A,B].$
	\end{enumerate}
\end{corollary}

In the next result, we find the conditions on the real numbers $A, B$, $\beta$ and $\gamma$ so that $p(z)\prec \sqrt{1+z}$, whenever $p(z)+(zp'(z))/(\beta p(z)+\gamma) \prec (1+Az)/(1+Bz)$, $-1\leq B<A\leq 1$, where $p \in \mathcal{H}$ with $p(0)=1.$ As an application of the next result, it provides sufficient conditions for $f \in \mathcal{A}$ to belong to the class $\mathcal{S}^*_L$.

\begin{theorem}\label{p7thm2.4}
	Let $-1\leq B<A\leq 1$ and $\beta, \gamma \in \mathbb{R}$ satisfy the following conditions:
	\begin{enumerate}[(i)]
		\item $ 1+4(\sqrt{2}-1)\beta-2(\sqrt{2}-2)\gamma \geq B(-2A(2 \beta+\sqrt{2}\gamma)+B(1+4(\sqrt{2}\beta+\gamma)))$.
		\item $(1+4(\sqrt{2}-1)\beta-2(\sqrt{2}-2)\gamma)^{2}\geq (-2A(2 \beta+\sqrt{2}\gamma)+B(1+4(\sqrt{2}\beta+\gamma)))^{2} .$
	\end{enumerate}
	Let $p \in \mathcal{H}$ with $p(0)=1$. If the function $p$ satisfies
	\[ p(z)+\frac{zp'(z)}{\beta p(z)+\gamma} \prec \frac{1+Az}{1+Bz}, \]
	then $p(z)\prec \sqrt{1+z}$.
\end{theorem}

\begin{proof}
	Define the functions $P$ and $w$ as follows:
	\begin{equation}\label{p7eqn2.13}
	P(z)=p(z)+\frac{zp'(z)}{\beta p(z)+\gamma}\quad \text{and}\quad w(z)=p^{2}(z)-1
	\end{equation}
	which implies $p(z)=\sqrt{1+w(z)}.$ Clearly, $w(z)$ is analytic in $\mathbb{D}$ with $w(0)=0$. In order to complete our proof, we need to show that $|w(z)|<1$ in $\mathbb{D}$. Assume that there exists $z_{0}\in \mathbb{D}$ such that
	\[\max_{|z|\leq |z_{0}|}|w(z)|=|w(z_{0})|=1,\]
	then by Lemma ~\ref{p7lem1}, it follows that there exists $k\geq 1$ so that $z_{0}w'(z_{0})=kw(z_{0}).$ Let $w(z_{0})=e^{it}$, $(-\pi\leq t \leq \pi)$. By using ~\eqref{p7eqn2.13}, we get
	\[P(z)=\sqrt{1+w(z)}+\frac{zw'(z)}{2\sqrt{1+w(z)}(\beta\sqrt{1+w(z)}+\gamma)}.\]
	A simple computation shows that
	\[P(z_{0})=\frac{k e^{i t}+2 \left(1+e^{i t}\right) \left(\gamma +\beta  \sqrt{1+e^{i t}}\right)}{2 \sqrt{1+e^{i t}} \left(\gamma +\beta  \sqrt{1+e^{i t}}\right)}\quad (-\pi\leq t \leq \pi)\]
and
	\begin{equation}\label{p7eqn2.14}
	\left|\frac{P(z_{0})-1}{A-BP(z_{0})}\right|^{2}=:\frac{f(t)}{g(t)}\quad (-\pi\leq t \leq \pi),
	\end{equation}
	where
	\begin{equation*}
	\begin{split}
	f(t)=&\big((2 \beta  \cos t+2 (\beta -\gamma )) \sin( \arg(1+e^{i t})/2) \sqrt{2 \cos (t/2)} \\
	&+\sin t (k+2 (\gamma +\beta  (-1+\cos ( \arg (1+e^{i t})/2) \sqrt{2 \cos (t/2)})))\big)^2\\
	&+\big(-\cos t (k+2 (\gamma +\beta  (-1+\cos ( \arg (1+e^{i t})/2) \sqrt{2 \cos (t/2)})))+2 \beta  \sin t\\
	& \sin (\arg(1+e^{i t})/2) \sqrt{2 \cos (t/2)}+2(\beta -\gamma ) (1-\cos (\arg (1+e^{i t})/2) \sqrt{2 \cos (t/2)})\big)^2 \\
	\end{split}
	\end{equation*}
	and
	\begin{equation*}
	\begin{split}
	g(t)=&\big(-2A (\beta  \sin t +\gamma \sin( \arg(1+e^{i t})/2) \sqrt{2 \cos (t/2)})+4B\beta \cos^{2}(t/2)\sqrt{2 \cos (t/2)} \\
	&\sin( \arg (1+e^{i t})/2)+B \sin t (k+2 \gamma +2\beta  \cos ( \arg (1+e^{i t})/2) \sqrt{2 \cos (t/2)})\big)^2\\
	&+\big(-4A \beta  \cos^{2}(t/2)+B(k+2 \gamma)\cos t + 2\gamma-2B\beta \sin t\sin( \arg (1+e^{i t})/2) \\
	&\sqrt{2 \cos (t/2)}+2(-A\gamma+B \beta \cos t+\beta) \cos( \arg(1+e^{i t})/2) \sqrt{2 \cos (t/2)} \big)^{2}.\\
	\end{split}
	\end{equation*}
	Define $h(t)=f(t)-g(t)$. Since $h(t)$ is an even function of $t$, we restrict to $0\leq t \leq \pi$. It can be easily verified that for both the cases $(i)$ and $(ii)$, the function $h(t)$ attains its minimum value either at $t=0$ or $t=\pi$. Note that for $k \geq 1$, $h(\pi)=(1-B^{2})k^{2}>0$ and
	\begin{equation}\label{p7eqn2.15}
	h(0)=(4 (\sqrt{2}-1) \beta -2 (\sqrt{2}-2) \gamma +k)^2-(B (4 (\sqrt{2} \beta +\gamma)+k)-2 A(2 \beta +\sqrt{2} \gamma))^2=:S(k).
	\end{equation}
	The function $S'$ is increasing as $S''(k)=2(1-B^{2})>0$ and therefore, the given hypothesis $(i)$ yields that $S'(k)\geq S'(1)=2(1+4(\sqrt{2}-1)\beta-2(\sqrt{2}-2)\gamma) -2B(-2A(2 \beta+\sqrt{2}\gamma)+B(1+4(\sqrt{2}\beta+\gamma))) \geq 0$ which gives that $S(k)\geq S(1)=(1+4(\sqrt{2}-1)\beta-2(\sqrt{2}-2)\gamma)^{2}-(-2A(2 \beta+\sqrt{2}\gamma)+B(1+4(\sqrt{2}\beta+\gamma)))^{2}$. Thus, the use of given condition $(ii)$ and ~\eqref{p7eqn2.15} yields $h(0)\geq 0.$ So, $h(t) \geq 0$ for all $t \in [0,\pi]$ and therefore, ~\eqref{p7eqn2.14} implies $|(P(z_{0})-1)/(A-BP(z_{0}))| \geq 1$. This contradicts the fact that $P(z)\prec (1+Az)/(1+Bz)$ and completes the proof.
\end{proof}

The next corollary is obtained by substituting $p(z)=zf'(z)/f(z)$ with $\gamma=0$, $A=1-2\alpha$, $(0 \leq \alpha <1)$ and $B=-1$ in Theorem ~\ref{p7thm2.4}.

\begin{corollary}\label{p7cor3.14}
	Let $f \in \mathcal{A}$. If the function $f$ satisfies the subordination
	\[\frac{z f'(z)}{f(z)} + \frac{1}{\beta}\left(1+\frac{zf''(z)}{f'(z)}-\frac{zf'(z)}{f(z)}\right)\prec \frac{1+(1-2\alpha)z}{1-z} \quad \left(\frac{1}{4(\alpha-\sqrt{2})}\leq \beta < 0, 0\leq \alpha <1\right),\]
	then $f \in \mathcal{S}^{*}_{L}.$
\end{corollary}

By taking $p(z)=zF'(z)/F(z)$ with $\beta=1$ and $\gamma=c$ in Theorem ~\ref{p7thm2.4} gives the following corollary:

\begin{corollary}\label{p7cor3.3}
	Let $-1\leq B<A\leq 1$ satisfy the following conditions:
	\begin{enumerate}[(i)]
		\item $ 1+4(\sqrt{2}-1)-2(\sqrt{2}-2)c \geq B(-2A(2+\sqrt{2}c)+B(1+4(\sqrt{2}+c)))$.
		\item $(1+4(\sqrt{2}-1)-2(\sqrt{2}-2)c )^{2}\geq (-2A(2+\sqrt{2}c)+B(1+4(\sqrt{2}+c)))^{2} .$
	\end{enumerate}
	If $f \in \mathcal{S}^*[A,B]$ then $F \in \mathcal{S}^*_L$.
\end{corollary}

By taking $p(z)=F'(z)$ with $\beta=0$ and $\gamma=c+1$ in Theorem ~\ref{p7thm2.4} gives the following corollary:

\begin{corollary}\label{p7cor3.4}
	Suppose that $-1\leq B<A\leq 1$ satisfy the following conditions:
	\begin{enumerate}[(i)]
		\item $ 5- 2\sqrt{2}-2(\sqrt{2}-2)c \geq B(-2\sqrt{2}(c+1)A+(5+4c)B)$.
		\item $(5- 2\sqrt{2}-2(\sqrt{2}-2)c)^{2}\geq (-2\sqrt{2}(c+1)A+(5+4c)B)^{2} .$
	\end{enumerate}
	If $f \in \mathcal{R}[A,B]$ then $F'(z) \prec \sqrt{1+z}$.
\end{corollary}

In the next result, we compute the conditions on the real numbers $A, B$, $\beta$ and $\gamma$ so that $p(z)+(zp'(z))/(\beta p(z)+\gamma) \prec (1+Az)/(1+Bz), (-1\leq B<A\leq 1)$ implies $p(z)\prec e^{z}$, where $p \in \mathcal{H}$ with $p(0)=1.$ As an application of the next result, it provides sufficient conditions for $f \in \mathcal{A}$ to belong to the class $\mathcal{S}^*_e$.

\begin{theorem}\label{p7thm2.5}
	Let $-1\leq B<A\leq 1$ and $\beta, \gamma \in \mathbb{R}$ satisfy the following conditions:
	\begin{enumerate}[(i)]
		\item $e^2\beta(1-B^2)+e(-B(-A\beta +B\gamma+B)-\beta +\gamma+1)+\gamma (A B-1)\geq 0$.
		\item $(e((A+e-1)\beta-(e\beta+1)B+1)+\gamma(A+e(1-B)-1))(e(-(A-e+1)\beta+B(e\beta+1)+1)+\gamma(-A+e (B+1)-1))\geq 0 .$
		\item $e(\beta(1-A B)+B^2(\gamma-1)-\gamma +1)+e^2\gamma(1-AB)+\beta (B^2-1)\geq 0.$
		\item $(e((A-1)\beta+(1-B)(\gamma-1))+e^2(A-1)\gamma+\beta(1-B))(-e((A+1) \beta+(B+1)(1-\gamma))-e^2(A+1)\gamma+\beta(B+1))\geq 0$.
	\end{enumerate}
	Let $p \in \mathcal{H}$ with $p(0)=1$. If the function $p$ satisfies
	\[ p(z)+\frac{zp'(z)}{\beta p(z)+\gamma} \prec \frac{1+Az}{1+Bz}, \]
	then $p(z)\prec e^{z}$.
\end{theorem}

\begin{proof}
	Define the functions $\psi: {\mathbb{C}}^{2} \times \mathbb{D} \to \mathbb{C}$ and $q:\mathbb{D} \to \mathbb{C}$ as follows:
	\begin{equation}\label{p7eqn2.16}
	\psi(r,s;z)=r+\frac{s}{\beta r+\gamma}\quad \text{and}\quad q(z)=\frac{1+Az}{1+Bz}
	\end{equation}
	so that $\Omega:=q(\mathbb{D})=\left\{w\in\mathbb{C}:|(w-1)/(A-Bw)|<1 \right\}$ and $\psi(p(z),zp'(z);z)\in \Omega$ for $z \in \mathbb{D}$. To prove $p(z)\prec e^{z}$, we use Lemma ~\ref{p7lem2} so we need to show that $\psi(e^{{e}^{it}}, k e^{it}e^{{e}^{it}}; z) \notin \Omega$ which is equivalent to show that $|(\psi(e^{{e}^{it}}, k e^{it}e^{{e}^{it}}; z)-1)/(A-B \psi(e^{{e}^{it}}, k e^{it}e^{{e}^{it}}; z))| \geq 1$, where $z \in \mathbb{D}$, $t \in [-\pi, \pi]$ and $k \geq 1$. A simple computation and ~\eqref{p7eqn2.16} yield that
	\begin{equation*}
	\psi(e^{{e}^{it}}, k e^{it}e^{{e}^{it}}; z)=e^{e^{i t}}+\frac{k e^{i t}e^{e^{i t}}}{\beta e^{e^{i t}}+\gamma}\quad (-\pi\leq t \leq \pi)
	\end{equation*}
	and
	\begin{equation}\label{p7eqn2.17}
	\left|\frac{\psi(e^{{e}^{it}}, k e^{it}e^{{e}^{it}}; z)-1}{A-B\psi(e^{{e}^{it}}, k e^{it}e^{{e}^{it}}; z)}\right|^{2}=:\frac{f(t)}{g(t)}\quad (-\pi\leq t \leq \pi),
	\end{equation}
	where
	\begin{equation*}
	\begin{split}
	f(t)=&e^{3 \cos t}(2 \beta  k \sin t \sin (\sin t)+2 \beta  k \cos t \cos (\sin t)-2 \beta ^2 \cos (\sin t)+2 \beta  \gamma  \cos (\sin t))\\
	&+e^{2 \cos t}((\beta-\gamma)^2+k^2-2 \beta  k \cos t+2 \gamma  k \cos t+2 \beta \gamma  \sin ^2(\sin t)-2\beta \gamma \cos ^2(\sin t))\\
	&+e^{\cos t}(2 \gamma  k \sin t \sin(\sin t)-2\gamma  k\cos t \cos (\sin t)+2 \beta  \gamma  \cos (\sin t)-2 \gamma^2 \cos(\sin t))\\
	&+\beta ^2 e^{4 \cos t}+\gamma ^2\\
	\end{split}
	\end{equation*}
	and
	\begin{equation*}
	\begin{split}
	g(t)=&A^2\gamma^2+\beta^2 B^2 e^{4 \cos t}+2\beta B e^{3 \cos t}((B\gamma -A\beta )\cos (\sin t)+Bk \cos (t-\sin t))\\
	&+e^{2 \cos t}(B(B(\gamma ^2+k^2)-2 A \beta\gamma)+2B(B\gamma-A\beta)k \cos t-2AB\beta\gamma \cos(2\sin t)\\
	&+A^2 \beta ^2)+2A\gamma  e^{\cos t}((A\beta-B\gamma)\cos (\sin t)-Bk\cos(t+\sin t)).
	\end{split}
	\end{equation*}
	Define $h(t)=f(t)-g(t)$. Since $h(-t)=h(t)$, we restrict to $0\leq t \leq \pi$. It can be easily verified that the function $h(t)$ attains its minimum value either at $t=0$ or $t=\pi$. For $k \geq 1$, we have
	\begin{equation}\label{p7eqn2.18}
	\begin{split}
	h(0)=&e^2((1-A^2)\beta^2+2k(\beta (AB-1)+(1-B^2)\gamma)+4\beta\gamma(AB-1)+(1-B^2)(\gamma^2+k^2))\\
	&+2e\gamma(-A^2\beta+(AB-1)(\gamma+k)+\beta)+2e^3\beta(\beta(AB-1)+(1-B^2)(\gamma+k))\\
	&+e^4\beta^2(1-B^2)+(1-A^2)\gamma^{2}=:\phi(k)
	\end{split}
	\end{equation}
	and
	\begin{equation}\label{p7eqn2.19}
	\begin{split}
	h(\pi)=&\frac{-1}{e^4}(e ((A-1)\beta  +(1-B)(\gamma-k))+e^2(A-1)\gamma+\beta(1-B))\\
	&(e((1+A)\beta+(B+1)(k-\gamma))+e^2 (A+1)\gamma -\beta  (B+1))=:\psi(k).
	\end{split}
	\end{equation}
	The function $\phi'$ is increasing as $\phi''(k)=2(1-B^{2})e^2>0$ and therefore, the given hypothesis $(i)$ yields that $\phi'(k)\geq \phi'(1)=2e(e(-B(-A\beta +B\gamma+B)-\beta +\gamma+1)+\gamma (A B-1)+e^2\beta(1-B^2))\geq 0$ which gives that $\phi(k)\geq \phi(1)=(e((A+e-1)\beta-(e\beta+1)B+1)+\gamma(A+e(1-B)-1))(e(-(A-e+1)\beta+B(e\beta+1)+1)+\gamma(-A+e (B+1)-1))$. Thus, the use of given condition $(ii)$ and ~\eqref{p7eqn2.18} yields $h(0)\geq 0.$
	
	In view of $(iii)$, observe that $\psi''(k)=2(1-B^2)/e^2>0$ and therefore, $\min \psi'(k)=\psi'(1)=2(e(\beta(1-A B)+B^2(\gamma-1)-\gamma +1)+e^2\gamma(1-AB)+\beta (B^2-1))/e^3 \geq 0$ which implies $\min \psi(k)=\psi(1)=((e((A-1)\beta+(1-B)(\gamma-1))+e^2(A-1)\gamma+\beta(1-B))(-e((A+1) \beta+(B+1)(1-\gamma))-e^2(A+1)\gamma+\beta(B+1)))/e^4$. Hence, the use of given condition $(iv)$ and ~\eqref{p7eqn2.19} yields that $h(\pi)\geq 0.$ So, $h(t)\geq 0, (0\leq t \leq \pi)$ and thus, ~\eqref{p7eqn2.17} implies $|(\psi(e^{{e}^{it}}, k e^{it}e^{{e}^{it}}; z)-1)/(A-B\psi(e^{{e}^{it}}, k e^{it}e^{{e}^{it}}; z))| \geq 1$ and therefore, $p(z) \prec e^{z}$.
\end{proof}

The next corollary is obtained by substituting $p(z)=zf'(z)/f(z)$ with $\gamma=0$, $B=0$ and $A=1-\alpha$, $(0 \leq \alpha<1)$ in Theorem ~\ref{p7thm2.5}.

\begin{corollary}\label{p7cor3.15}
	Suppose $0 \leq \alpha<1$ and $\beta \geq 1/(1-e)$ satisfy the conditions $(-\alpha  \beta +\beta  e+1) (\beta  (\alpha +e-2)+1)\geq 0$ and
	$(\beta -e ((2-\alpha ) \beta +1)) (\beta +e (-\alpha  \beta -1))\geq 0$.
	If the function $f \in \mathcal{A}$ satisfies the condition
	\[\left|\frac{z f'(z)}{f(z)} + \frac{1}{\beta}\left(1+\frac{zf''(z)}{f'(z)}-\frac{zf'(z)}{f(z)}\right)-1\right|< 1-\alpha,\]
	then $f \in \mathcal{S}^{*}_{e}.$
\end{corollary}

The two parts of the following corollary are obtained by taking $p(z)=zF'(z)/F(z)$ with $\beta=1$, $\gamma=c$ and $p(z)=F'(z)$ with $\beta=0$, $\gamma=c+1$ respectively in Theorem ~\ref{p7thm2.5}.

\begin{corollary}\label{p7cor3.5}
	\begin{enumerate}[(i)]
		\item  If the function $f\in \mathcal{S}^*[A,B]$ and the conditions of the Theorem ~\ref{p7thm2.5} hold with $\beta=1$ and $\gamma=c$, then $F \in \mathcal{S}^*_e.$
		\item The function $f \in \mathcal{R}[A,B]$ and the conditions of the Theorem ~\ref{p7thm2.5} hold with $\beta=0$ and $\gamma=c+1$, then  $F'(z)\prec e^{z}.$
	\end{enumerate}
\end{corollary}

In the next result, we find the conditions on the real numbers $A, B$, $\beta$ and $\gamma$ so that $p(z)\prec (1+Az)/(1+Bz), (-1\leq B<A\leq 1)$, whenever $p(z)+(zp'(z))/(\beta p(z)+\gamma) \in \Omega_{P}$, where $p \in \mathcal{H}$ with $p(0)=1.$ As an application of the next result, it provides sufficient conditions for $f \in \mathcal{A}$ to belong to the class $\mathcal{S}^*[A,B]$.

\begin{theorem}\label{p7thm2.3}
	Let $-1\leq B<A\leq 1$ and $\beta, \gamma \in \mathbb{R}$. For $k\geq 1$ and $0\leq m \leq 1$, assume that $G:=A \beta + B \gamma, L:= k+ \beta + \gamma$. Further assume that
	\begin{enumerate}[(i)]
		\item $BG(\beta + \gamma)>0.$
		\item
		\begin{equation*}
		\begin{split}
		\big(G&(A^2 L+4 (\beta +\gamma ))-2 B(A G L+2 (\beta +\gamma )^2+2 G^2)+B^2 G (4 (\beta +\gamma )+L)\big)\\
		&\big(G \left(A^2 L-4 (\beta +\gamma )\right)+B \left(-2 A G L+4 (\beta +\gamma )^2+4 G^2\right)+B^2 G (L-4 (\beta +\gamma ))\big)\\
		&\quad{}\geq 2 G (A-B)^2 \big(G L(A^2 L-4 (\beta +\gamma ))-2 B(A G L^2+2 G^2 (L-2 (\beta +\gamma ))\\
		&\qquad{}-2 L (\beta +\gamma ) (-\beta -\gamma +2 L))+B^2 G L (L-4 (\beta +\gamma ))\big).
		\end{split}
		\end{equation*}
		\item $8G (A-B)^2 (\beta +\gamma +k)\leq 2 (B-1)^{2} G (\beta +\gamma )+2 B (\beta +\gamma -G )^2.$
		\item $1+\beta+\gamma\geq 0$, $G \geq 0.$
		\item $4 m^4 (A-B)^2 (\beta +\gamma +G+1)^2 \geq(B+1)^2 (\beta +\gamma +G)^2.$
	\end{enumerate}
	Let $p \in \mathcal{H}$ with $p(0)=1$. If the function $p$ satisfies
	\[ p(z)+\frac{zp'(z)}{\beta p(z)+\gamma} \prec \varphi_{PAR}(z), \]
	then $p(z)\prec (1+Az)/(1+Bz)$.
\end{theorem}

\begin{proof}
	Define the functions $P$ and $w$ as given by the equation ~\eqref{p7eqn2.4} which implies $p(z)=(1+Aw(z))/(1+Bw(z)).$ Proceeding as in Theorem ~\ref{p7thm2.2}, we need to show that $|w(z)|<1$ in $\mathbb{D}$. If possible suppose that there exists $z_{0}\in \mathbb{D}$ such that
	\[\max_{|z|\leq |z_{0}|}|w(z)|=|w(z_{0})|=1,\]
	then by Lemma ~\ref{p7lem1}, it follows that there exists $k\geq 1$ so that $z_{0}w'(z_{0})=kw(z_{0}).$ Let $w(z_{0})=e^{it}$, $(-\pi\leq t \leq \pi)$. A simple calculation and by using ~\eqref{p7eqn2.4}, we get
	\begin{equation}\label{p7eqn2.80}
	P(z_{0})=\frac{k e^{i t} (A-B)+\left(1+A e^{i t}\right) \left(\beta +\gamma +G e^{i t}\right)}{\left(1+B e^{i t}\right) \left(\beta +\gamma +G e^{i t}\right)} \quad (-\pi \leq t \leq \pi).
	\end{equation}
	Define the function $h$ by
	\begin{equation}\label{p7eqn2.10}
	h(z)=u+iv=\sqrt{(P(z)-1)\pi^{2}/2}.
	\end{equation}
We show that $|(e^{h(z_{0})}-1)/(e^{h(z_{0})}+1)|^{2}\geq 1$; this condition  is same as the inequality $\RE e^{h(z_{0})}\leq 0$. This last inequality is indeed equivalent to  $\cos v \leq 0$ or $1/2\leq|v/\pi|\leq 1$. By using the definition of $h$ given in ~\eqref{p7eqn2.10} together with ~\eqref{p7eqn2.80}, we get
	\begin{equation}\label{p7eqn2.60}
	\frac{|v|}{\pi}=\frac{\sqrt{A-B}| m(t)| |G e^{i t}+L|^{1/2}}{\sqrt{2}|1+B e^{i t}|^{1/2}|G e^{i t}+\beta +\gamma|^{1/2}} \quad (-\pi \leq t \leq \pi),
	\end{equation}
	where $m(t)=\sin \left( \arg \left((e^{i t} (A-B) (G e^{i t}+L))/((1+B e^{i t})(G e^{i t}+\beta +\gamma ))\right)/2\right).$
	
	(a) We will first show that $|v/\pi|\leq 1$ which by using the fact that $|m(t)|\leq 1$ and ~\eqref{p7eqn2.60} is same as to show that $f(t)\geq 0$ $(-\pi \leq t \leq \pi)$, where
	\[ f(t)=4(1+B^{2}+2B \cos t)((\beta +\gamma)^{2}+G^{2}+2(\beta +\gamma)G \cos t)-(A-B)^{2}(L^{2}+G^{2}+2LG\cos t).\]
	After substituting $x=\cos t$ $(-\pi \leq t \leq \pi)$, the above inequality reduces to $F(x)\geq 0$ for all $x$ with $-1\leq x \leq 1$, where
	\[F(x)=4(1+B^{2}+2Bx)((\beta +\gamma)^{2}+G^{2}+2(\beta +\gamma)G x)-(A-B)^{2}(L^{2}+G^{2}+2LGx). \]
	A simple computation shows that for
	\[x_{0}=\frac{G \left(A^2 L-4 (\beta +\gamma )\right)-2 B \left(A G L+2 (\beta +\gamma )^2+2 G^2\right)+B^2 G (L-4 (\beta +\gamma ))}{16 BG(\beta+\gamma)},\] $F'(x_{0})=0$ and $F''(x_{0})=32 BG(\beta+\gamma)>0$ by the given condition $(i)$. Therefore, $F(x)\geq F(x_{0})$. Observe that
	\begin{equation*}
	\begin{split}
	F(x_{0})=\frac{1}{16 B G (\beta +\gamma )}&\big(\big(G \left(A^2 L+4 (\beta +\gamma )\right)-2 B \left(A G L+2 (\beta +\gamma )^2+2 G^2\right)\\
	&+B^2 G (4 (\beta +\gamma )+L)\big) \big(G \left(A^2 L-4 (\beta +\gamma )\right)+B^2 G (L-4 (\beta +\gamma ))\\
	& +B \left(-2 A G L+4 (\beta +\gamma )^2+4 G^2\right)\big)-2 G (A-B)^2 \\
	&\big(G L \left(A^2 L-4 (\beta +\gamma )\right)+B^2 G L (L-4 (\beta +\gamma ))\\
	&-2 B \left(A G L^2+2 G^2 (L-2 (\beta +\gamma ))-2 L (\beta +\gamma ) (-\beta -\gamma +2 L)\right)\big)\big)
	\end{split}
	\end{equation*}
	and $F(x_{0})\geq 0$ by the given condition $(ii)$.
	
	(b) We will next show that $|v/\pi|\geq 1/2$ which by using ~\eqref{p7eqn2.60} is same as to show that $g(t)\geq 0$ $(-\pi \leq t \leq \pi)$, where
	\[ g(t)=4(A-B)^{2}m^{4}(t)(L^{2}+G^{2}+2LG\cos t)-(1+B^{2}+2B \cos t)((\beta +\gamma)^{2}+G^{2}+2(\beta +\gamma)G \cos t)\]
	After substituting $x=\cos t$ $(-\pi \leq t \leq \pi)$ and $m=m(t)$, the above inequality reduces to $H(x)\geq 0$ for all $x$ with $-1\leq x \leq 1$, where
	\[H(x)=4(A-B)^{2}m^{4}(L^{2}+G^{2}+2LGx)-(1+B^{2}+2Bx)((\beta +\gamma)^{2}+G^{2}+2(\beta +\gamma)G x). \]
	In view of $(i)$, $(iii)$, $(iv)$ and the fact that $-1 \leq m \leq 1$, we see that $H''(x)=-8BG(\beta+\gamma)<0$ and hence $H'(x)\leq H'(-1)=8 m^4 G (A-B)^2 (\beta +\gamma +k)-2 (B-1)^{2} G (\beta +\gamma )-2 B (-G+\beta +\gamma )^2\leq 0$. Thus, $H(x)\geq H(1)=4 m^4 (A-B)^2 (\beta +\gamma +G+k)^2-(B+1)^2 (\beta +\gamma +G)^2=:\psi(k)$. Using $(iv)$, we observe that $\psi''(k)=8 m^4 (A-B)^2 \geq 0$ and hence for $k\geq 1$, we have $\psi'(k)\geq \psi'(1)=8 m^4 (A-B)^2 (\beta +\gamma +G+1)\geq 0$. Thus by using $(v)$, we get $H(x)\geq \psi(k)\geq \psi(1)=4 m^4 (A-B)^2 (\beta +\gamma +G+1)^2-(B+1)^2 (\beta +\gamma +G)^2 \geq 0$. This completes the proof.
\end{proof}

The next corollary is obtained by substituting $p(z)=zf'(z)/f(z)$ with $\gamma=0$, $B=-1$ and $A=1-2\alpha$, $(0 \leq \alpha<1)$ in Theorem ~\ref{p7thm2.3}.

\begin{corollary}\label{p7cor3.17}
	Let $1/2 < \alpha<1$, $-1 \leq \beta<0$ and $k \geq1$ satisfy the conditions $(2 \alpha ^2+\alpha -3)^2 \beta ^2+(4 \alpha ^4-12 \alpha ^3+13 \alpha ^2+2 \alpha -3) k^2+2 (4 \alpha ^4-20 \alpha ^3+17 \alpha ^2+2 \alpha -3) \beta  k \leq 0$ and $(\alpha ^2+2 \alpha -1) \beta ^2\leq 4 (\alpha -1)^2 (2 \alpha -1) \beta  (\beta +k)$. If the function $f \in \mathcal{A}$ satisfies the subordination
	\[\frac{z f'(z)}{f(z)} + \frac{1}{\beta}\left(1+\frac{zf''(z)}{f'(z)}-\frac{zf'(z)}{f(z)}\right)\prec \varphi_{PAR}(z),\]
	then $f \in \mathcal{S}^{*}(\alpha).$
\end{corollary}

\end{document}